\documentclass[12pt]{amsart}
\usepackage{amsmath,amssymb,amsthm,booktabs,geometry,microtype}
\usepackage[colorlinks,linkcolor=blue,citecolor=blue,urlcolor=blue]{hyperref}
\usepackage{enumerate}

\newtheorem{theorem}{Theorem}[section]
\newtheorem{lemma}[theorem]{Lemma}
\newtheorem{proposition}[theorem]{Proposition}
\newtheorem{corollary}[theorem]{Corollary}
\theoremstyle{definition}
\newtheorem{observation}[theorem]{Computational observation}
\newtheorem{definitionx}[theorem]{Definition}
\newtheorem{example}[theorem]{Example}
\theoremstyle{remark}
\newtheorem{remark}[theorem]{Remark}

\DeclareMathOperator{\den}{den}
\newcommand{\Zt}{\mathbb Z_2}
\newcommand{\Qt}{\mathbb Q_2}
\newcommand{\ma}{\mu_\alpha}
\newcommand{\mb}{\mu_\beta}
\newcommand{\Nn}{\mathcal N}
\newcommand{\Ber}{\mathcal B}
\newcommand{\Eul}{\mathcal E}
\newcommand{\doi}[1]{\href{https://doi.org/#1}{\nolinkurl{doi:#1}}}

\title[Dyadic denominators of the Jacobi phase constants]{Dyadic denominator bounds and a transfer principle for the phase constants of Jacobi zeros}

\author[I. Area]{Iv\'an Area}
\address[I. Area]{IFCAE, Universidade de Vigo, Departamento de Matem\'atica Aplicada II, E. E. Aeron\'au\-tica e do Espazo, Campus As Lagoas s/n, 32004 Ourense, Spain}
\email[I. Area]{area@uvigo.gal}



\begin{document}

\begin{abstract}
A formal large-degree phase expansion associated with the zeros of the Jacobi polynomial $P_n^{(\alpha,\beta)}$ contains a sequence of constants $\kappa_1,\kappa_2,\dots$, each of which is a polynomial in $\alpha^2$ and $\beta^2$ with rational coefficients; for instance $\kappa_1=(\alpha^2-\beta^2)/4$.  This paper asks a simple arithmetic question: which primes, and to which powers, occur in the denominators of these coefficients?  We prove that the odd part of the common denominator of $\kappa_r$ divides $\operatorname{lcm}(1,3,\ldots,2r-1)$, and that the power of $2$ is at most $2^{E_r}$ with $E_r=3r-1+\nu_2((r-1)!)$, where $\nu_2$ is the $2$-adic valuation.  The study of the leading coefficient leads to the formula
\[
 \nu_2\!\left(\sum_{j=0}^{m}\binom mj\frac1{2j+1}\right)=m+\nu_2(m+1)\qquad(m\ge0),
\]
which we prove by showing that the series $\sum_{k\ge0}k!/(2k+1)!!$, whose real sum is $\pi/2$, converges to $0$ in the field $\mathbb Q_2$ of $2$-adic numbers.  Finally, we show that the statement ``the exponent $E_r$ is attained for every $r$'' is equivalent to an integrality property of the Taylor coefficients of a single explicit power series in one variable.  That property is established in a separate paper, and is confirmed here independently, by exact computation, for $r\le48$.
\end{abstract}

\maketitle

\section{Introduction}

Let $P_n^{(\alpha,\beta)}$ be the Jacobi polynomial of degree $n$ with parameters $\alpha,\beta>-1$, in the standard normalisation of~\cite{Szego,DLMF}.  Its $n$ zeros lie in $(-1,1)$; we write them as $x_{n,k}=\cos\theta_{n,k}$ with $0<\theta_{n,1}<\dots<\theta_{n,n}<\pi$.  Throughout the paper
\[
  \Nn=n+\frac{\alpha+\beta+1}2,\qquad A=\alpha^2,\qquad B=\beta^2 .
\]
The classical first-order description of the zeros (cf.~\cite[Ch.~8]{Szego}) is $\Nn\theta_{n,k}\approx\bigl(k+\frac\alpha2-\frac14\bigr)\pi$.  It is the first term of a complete asymptotic expansion: transforming the Jacobi differential equation to Liouville normal form and solving the associated phase equation recursively in powers of $\Nn^{-2}$, one obtains a \emph{phase function}
\begin{equation}\label{eq:phase}
  \Xi(\theta)=\Nn\theta-\Bigl(\frac\alpha2+\frac14\Bigr)\pi+\sum_{r\ge1}\frac{\Psi_r(\theta)}{\Nn^{2r-1}},
\end{equation}
such that the zeros are, to all orders in $1/\Nn$, the solutions of $\Xi(\theta_{n,k})=(k-\frac12)\pi$.  Each $\Psi_r$ is obtained by an integration in $\theta$, and is therefore determined only up to an additive constant.  With the natural choice of primitive (the one that is an odd Laurent polynomial in $\cot(\theta/2)$), $\Psi_r$ is a Laurent polynomial in $\cot(\theta/2)$ whose coefficients are polynomials in $A$ and $B$ with rational coefficients.  The additive constant is then recorded by the value at the midpoint $\theta=\pi/2$ of the interval,
\[
  \kappa_r=\Psi_r(\pi/2)\in\mathbb Q[A,B].
\]
We call $\kappa_1,\kappa_2,\dots$ the \emph{phase constants}.  The first three are
\begin{align*}
  \kappa_1&=\frac{A-B}4,\qquad
  \kappa_2=\frac{-7(A-B)+2(A^2-B^2)}{96},\\
  \kappa_3&=\frac{781(A-B)}{7680}-\frac{19(A^2-B^2)}{384}+\frac{7(A^3-B^3)}{1920}+\frac{AB(A-B)}{128}.
\end{align*}
Each $\kappa_r$ changes sign when $A$ and $B$ are exchanged (Remark~\ref{rem:antisym}), so it is divisible by $A-B$ in $\mathbb Q[A,B]$.  In particular the phase constants vanish when $\alpha=\pm\beta$, which is why they do not appear in the classical theory of Legendre and Gegenbauer polynomials.

Let $\den\kappa_r$ denote the least positive integer $d$ such that $d\,\kappa_r$ has integer coefficients.  The first values are shown in Table~\ref{tab:pattern}.

\begin{table}[h]
\centering
\begin{tabular}{@{}cllcc@{}}
\toprule
$r$ & $\den\kappa_r$ & factorisation & $E_r=3r-1+\nu_2((r-1)!)$ & $\operatorname{lcm}(1,3,\dots,2r-1)$\\
\midrule
$1$ & $4$ & $2^{2}$ & $2$ & $1$\\
$2$ & $96$ & $2^{5}\cdot3$ & $5$ & $3$\\
$3$ & $7680$ & $2^{9}\cdot15$ & $9$ & $15$\\
$4$ & $86016$ & $2^{12}\cdot21$ & $12$ & $105$\\
\bottomrule
\end{tabular}
\caption{Denominators of the first four phase constants.  The power of $2$ is exactly $2^{E_r}$, and the odd part divides $\operatorname{lcm}(1,3,\dots,2r-1)$.}
\label{tab:pattern}
\end{table}

\noindent The question studied in this paper is: \emph{what are the prime factors of $\den\kappa_r$, and in particular what is the exact power of $2$?}  The odd part turns out to be easy.  The power of $2$ is the substantial part: its exponent $2,5,9,12,\dots$ is given by $E_r=3r-1+\nu_2((r-1)!)$.  The purpose of the present paper is to prove that $E_r$ is an upper bound for every $r$, to isolate the exact obstruction to its sharpness, and to verify sharpness in a substantial finite range; the proof of sharpness for every $r$ is completed in~\cite{Ultraspherical}.

\medskip

We first collect everything that is used from elementary $2$-adic arithmetic; a standard reference is~\cite{Koblitz}.

\begin{itemize}
\item For a nonzero rational number $q=2^e\,a/b$ with $a,b$ odd integers, $\nu_2(q)=e$ is the \emph{$2$-adic valuation} of $q$; $\nu_2(0)=+\infty$.  It satisfies $\nu_2(pq)=\nu_2(p)+\nu_2(q)$ and $\nu_2(p+q)\ge\min\bigl(\nu_2(p),\nu_2(q)\bigr)$, \emph{with equality when the two valuations are different}.  This last rule is used repeatedly.
\item $\Qt$ is the field of $2$-adic numbers, the completion of $\mathbb Q$ for the absolute value $|q|_2=2^{-\nu_2(q)}$, and $\Zt=\{q\in\Qt:\nu_2(q)\ge0\}$ is the ring of $2$-adic integers.  A rational number belongs to $\Zt$ exactly when its denominator is odd.  The \emph{units} of $\Zt$, written $\Zt^\times$, are the elements of valuation $0$.  A series $\sum_kc_k$ converges in $\Qt$ if and only if $\nu_2(c_k)\to+\infty$.
\item $s_2(m)$ is the number of digits $1$ in the binary expansion of the non-negative integer $m$.  It is subadditive: $s_2(a+b)\le s_2(a)+s_2(b)$.
\item \emph{Legendre's formula}: $\nu_2(m!)=m-s_2(m)$.
\item \emph{Kummer's theorem}~\cite{Kummer,Mihet}: $\nu_2\binom{2m}m=s_2(m)$.
\item $[t^k]f(t)$ is the coefficient of $t^k$ in a power series or polynomial $f$; similarly $[A^iB^j]$ for polynomials in two variables.
\item $(a)_k=a(a+1)\cdots(a+k-1)$ is the Pochhammer symbol, and $(2k+1)!!=1\cdot3\cdot5\cdots(2k+1)$.
\item $\Ber_{2k}$ are the Bernoulli numbers ($\Ber_2=\frac16$, $\Ber_4=-\frac1{30},\dots$).  By the von Staudt--Clausen theorem~\cite{Clausen,Rado,Carlitz}, $\nu_2(\Ber_{2k})=-1$ for $k\ge1$.
\end{itemize}

\noindent Three quantities built from these appear throughout:
\begin{equation}\label{eq:HEomega}
  H_r=r+\nu_2\bigl((r-1)!\bigr),\qquad
  E_r=(2r-1)+H_r=3r-1+\nu_2\bigl((r-1)!\bigr),\qquad
  \omega(s)=\Bigl\lceil\frac s2\Bigr\rceil+\nu_2(s).
\end{equation}
It is convenient to use, next to $A$ and $B$, the shifted variables
\[
  \ma=A-\tfrac14=\alpha^2-\tfrac14,\qquad \mb=B-\tfrac14=\beta^2-\tfrac14 ,
\]
which are the combinations in which the parameters enter the differential equation.  A polynomial in $(A,B)$ is also a polynomial in $(\ma,\mb)$, and conversely.

\medskip

Next we present the main results of the paper, which shall be proved in the following sections.  Precisely, $\kappa_r$ is defined in this paper by formula~\eqref{eq:kappafromL} of Section~\ref{sec:setting}; its relation with~\eqref{eq:phase} is discussed at the end of this Introduction.  The first two theorems bound the denominator.

\begin{theorem}[odd part of the denominator]\label{thm:odd}
For every $r\ge1$, the odd part of $\den\kappa_r$ divides $\operatorname{lcm}(1,3,\dots,2r-1)$.
\end{theorem}

\begin{theorem}[power of $2$: upper bound]\label{thm:coarseintro}
For every $r\ge1$, all coefficients of the polynomial $2^{E_r}\kappa_r$, written in the variables $\ma,\mb$, are $2$-adic integers: $2^{E_r}\kappa_r\in\Zt[\ma,\mb]$.
\end{theorem}

The next two theorems are arithmetic statements of independent interest.  They arise from the coefficient of the highest power $\ma^r$ in $\kappa_r$, which is a multiple of the sum in Theorem~\ref{thm:extremalintro}.

\begin{theorem}[a $2$-adic vanishing]\label{thm:vanishintro}
The series
\[
\sum_{k\ge0}\frac{k!}{(2k+1)!!} =\sum_{k\ge0}\frac{2^k}{(2k+1)\binom{2k}{k}}
\]
converges in $\Qt$, and its sum there is $0$.
\end{theorem}

In the real numbers the same series converges to $\pi/2$.  Indeed, $k!/(2k+1)!!=2^{-k}\int_0^1(1-u^2)^k\,du$, and summing the geometric series under the integral sign gives $\int_0^12/(1+u^2)\,du=\pi/2$.

\begin{theorem}[valuation of a binomial sum]\label{thm:extremalintro}
For every integer $m\ge0$,
\[
  \nu_2\Bigl(\sum_{j=0}^{m}\binom mj\frac1{2j+1}\Bigr)=m+\nu_2(m+1).
\]
\end{theorem}

For example, $m=3$ gives $1+1+\frac35+\frac17=\frac{96}{35}$ and $\nu_2(96)=5=3+\nu_2(4)$.

The last result concerns the question whether the exponent $E_r$ of Theorem~\ref{thm:coarseintro} is actually attained, as it is in Table~\ref{tab:pattern}.  The relevant coefficient is that of the monomial $A$ in $\kappa_r$,
\begin{equation}\label{eq:elldef}
  \ell_r=[A^1B^0]\,\kappa_r=\frac{\partial\kappa_r}{\partial A}\Big|_{A=B=0},
  \qquad \ell_1=\frac14,\quad \ell_2=-\frac7{96},\quad \ell_3=\frac{781}{7680},\quad \ell_4=-\frac{28421}{86016}.
\end{equation}
Since $A=B=0$ is the Legendre case $\alpha=\beta=0$, the number $\ell_r$ measures the first-order response of $\kappa_r$ when one moves away from the Legendre polynomials in the direction of~$A$.  We call
\begin{equation}\label{eq:sharp}
  \nu_2(\ell_r)=-E_r\qquad\text{for every }r\ge1
\end{equation}
the \emph{sharp law}.  The difficulty in proving~\eqref{eq:sharp} is that $\ell_r$ arises as a sum of very many terms, and the valuation of the sum is strictly larger than the smallest valuation among its terms: $2$-adic cancellation takes place.  Our result is that~\eqref{eq:sharp} can be restated, without any loss, as a property of one power series.

\begin{theorem}[transfer to one power series]\label{thm:transferintro}
Let $g_n\in\mathbb Q(i)$, $n\ge1$, be the numbers defined in~\eqref{eq:gdef} below, so that $\ell_r=2^{-(2r-1)}\operatorname{Im}g_{2r-1}$, and define the even power series
\begin{equation}\label{eq:Wintro}
  \sum_{m\ge0}w_mt^{2m}=\text{even part of }\ \frac{\sinh 2t}{t^2}\sum_{n\ge1}\frac{\operatorname{Im}g_n}{(n-1)!}\,t^n
  =1+\frac{t^2}{12}+\frac{t^4}{64}-\frac{143\,t^6}{80640}+\cdots .
\end{equation}
Then, for every $r\ge1$,
\[
  \nu_2(\ell_j)\ge-E_j\ \text{ for all }j\le r
  \quad\Longleftrightarrow\quad
  \nu_2(w_m)\ge-2\nu_2\bigl((2m)!\bigr)\ \text{ for all }m\le r-1,
\]
and, when these inequalities hold, $\nu_2(\ell_r)=-E_r$ if and only if $\nu_2(w_{r-1})=-2\nu_2((2r-2)!)$.  Consequently the sharp law~\eqref{eq:sharp} is equivalent to
\begin{equation}\label{eq:W}
  \bigl((2m)!\bigr)^2\,w_m\in\Zt^\times\qquad\text{for every }m\ge0 .\tag{W}
\end{equation}
\end{theorem}

Property~\eqref{eq:W} is proved in the companion paper~\cite{Ultraspherical}, by an analytic computation of the series~\eqref{eq:Wintro} in closed form; together with Theorem~\ref{thm:transferintro} this establishes the sharp law~\eqref{eq:sharp} for every $r$.  No result of the present paper depends on~\cite{Ultraspherical}.  Independently of it, we verify~\eqref{eq:W} here by exact computation for $m\le46$, that is, by Theorem~\ref{thm:transferintro}, the sharp law for $r\le47$; a direct computation of $\ell_{48}$ gives the case $r=48$ as well (Section~\ref{sec:evidence}).

\medskip

We close this Introduction by making precise the status of the statements and the organisation of the paper.  Two statements in this paper rest on exact computation in a finite range rather than on a proof valid for all $r$; we state them here so that there is no ambiguity.

\begin{enumerate}[(a)]
\item \emph{The definition of $\kappa_r$.}  Our proofs use the definition~\eqref{eq:kappafromL} of $\kappa_r$ in terms of the logarithm of an explicit double hypergeometric series.  That this coincides with the constant $\Psi_r(\pi/2)$ of the phase expansion~\eqref{eq:phase} has been checked in exact rational arithmetic for $r\le8$, against an independent computation based on the phase recursion.  All theorems are statements about the numbers defined by~\eqref{eq:kappafromL}.  Accordingly, their interpretation as statements about the constants of the phase expansion~\eqref{eq:phase} for $r>8$ is conditional on the identification between~\eqref{eq:kappafromL} and~\eqref{eq:phase} for those $r$.
\item \emph{The leading coefficient of $\kappa_r$.}  The identification of the coefficient of $\ma^r$ in $\kappa_r$ with a multiple of the binomial sum of Theorem~\ref{thm:extremalintro} (formula~\eqref{eq:toplayer}) has been checked exactly for $r\le20$.  Theorems~\ref{thm:vanishintro} and~\ref{thm:extremalintro} themselves are proved for all $m$ and do not depend on it.
\end{enumerate}

\noindent Everything labelled Theorem, Lemma, Proposition or Corollary is proved in full, with any dependence on~(b) stated explicitly; statements obtained by computation are labelled ``Computational observation''.  All computations are exact (rational or Gaussian-rational arithmetic); no floating-point arithmetic has been used.

Section~\ref{sec:setting} defines $\kappa_r$, proves a division-free recurrence with integer coefficients, and deduces Theorems~\ref{thm:odd} and~\ref{thm:coarseintro}.  Section~\ref{sec:extremal} proves Theorems~\ref{thm:vanishintro} and~\ref{thm:extremalintro} and applies them to the leading coefficient; it can be read independently of the rest.  Section~\ref{sec:normalform} proves Theorem~\ref{thm:transferintro}.  Section~\ref{sec:evidence} reports the computations.  Appendix~\ref{sec:mellinsplit} records an integral representation related to~\cite{Ultraspherical}.

\section{The phase constants and an integral recurrence}\label{sec:setting}

We begin with the precise definition of the phase constants.  For $0<\theta<\pi$ put
\[
  x=-\frac1{e^{i\theta}-1},\qquad y=\frac1{e^{i\theta}+1},\qquad Z=2\Nn,
\]
and let
\[
  a_k(\nu)=\frac{(\frac12+\nu)_k(\frac12-\nu)_k}{(-2)^kk!} =\frac1{8^kk!}\prod_{j=0}^{k-1}\bigl(4\nu^2-(2j+1)^2\bigr)
\]
be, up to normalisation, the coefficients of Hankel's expansion of the Bessel functions~\cite[Sec.~10.17]{DLMF}.  The large-degree asymptotics of $P_n^{(\alpha,\beta)}(\cos\theta)$ is governed by the formal inverse-factorial series
\begin{equation}\label{eq:invfact}
  F=\sum_{n\ge0}\frac{C_n}{(Z-1)(Z-2)\cdots(Z-n)},\qquad   C_n=\sum_{l=0}^na_l(\alpha)\,a_{n-l}(\beta)\,(2x)^l(2y)^{n-l},
\end{equation}
which is our starting point.  Each term of~\eqref{eq:invfact} can be expanded in powers of $Z^{-1}$, so the logarithm of $F$ is a formal power series
\[
  L=\log F=\sum_{n\ge1}L_n\,Z^{-n},
\]
whose coefficients $L_n$ are polynomials in $x$ and $y$ without constant term, with coefficients in $\mathbb Q[A,B]$ (Theorem~\ref{thm:hornrec} below).  The corrections to the phase in~\eqref{eq:phase} are carried by the imaginary part of $L$.  At the midpoint $\theta=\pi/2$ we have $e^{i\theta}=i$ and
\[
  x_0=\frac{1+i}2,\qquad y_0=\frac{1-i}2 .
\]

\begin{definitionx}\label{def:kappa}
For $r\ge1$,
\begin{equation}\label{eq:kappafromL}
  \kappa_r=2^{-(2r-1)}\operatorname{Im}L_{2r-1}(x_0,y_0)\ \in\ \mathbb Q[A,B],
\end{equation}
where the imaginary part is taken treating $A$ and $B$ as real variables.
\end{definitionx}

\begin{remark}[antisymmetry]\label{rem:antisym}
The series~\eqref{eq:invfact} is unchanged when $(\alpha,x)$ and $(\beta,y)$ are interchanged, so $L_n(x,y;A,B)=L_n(y,x;B,A)$.  The coefficients of $L_n$ are rational and $y_0=\overline{x_0}$, hence $L_n(x_0,y_0;A,B)=\overline{L_n(x_0,y_0;B,A)}$ for real $A,B$.  Taking imaginary parts, $\kappa_r(B,A)=-\kappa_r(A,B)$.
\end{remark}

The factor $2^{-(2r-1)}$ converts powers of $Z^{-1}=(2\Nn)^{-1}$ into powers of $\Nn^{-1}$.  As explained in item~(a) of the Introduction, the agreement of~\eqref{eq:kappafromL} with the constants of~\eqref{eq:phase} has been verified for $r\le8$; the values of $\kappa_1,\kappa_2,\kappa_3$ displayed in the introduction are obtained from both.

\medskip

The key to the arithmetic of $\kappa_r$ is that the logarithmic derivatives of $F$ satisfy a recurrence in which no division occurs.  Write
\[
  \Theta_x=x\frac{\partial}{\partial x},\qquad \Theta_y=y\frac{\partial}{\partial y},\qquad
  U=\Theta_xL,\qquad V=\Theta_yL,\qquad T=U+V,
\]
and expand $U=\sum_{n\ge1}U_nZ^{-n}$, $V=\sum_{n\ge1}V_nZ^{-n}$, $T_n=U_n+V_n$.  Note that $\Theta_x$ and $\Theta_y$ multiply the monomial $x^cy^d$ by $c$ and by $d$ respectively.

\begin{theorem}[integral recurrence]\label{thm:hornrec}
We have $U_1=x\ma$, $V_1=y\mb$, and for $n\ge1$
\begin{equation}\label{eq:hornrecc}
  U_{n+1}=\Theta_xT_n+\sum_{\substack{a+b=n\\ a,b\ge1}}U_aT_b -x\Bigl(\Theta_xU_n+U_n+\sum_{\substack{a+b=n\\ a,b\ge1}}U_aU_b\Bigr),
\end{equation}
and the same formula for $V_{n+1}$ with the roles of $(x,U,\ma)$ and $(y,V,\mb)$ exchanged.  Consequently, for every $n\ge1$:
\begin{enumerate}[\rm(i)]
\item $U_n$ and $V_n$ are polynomials in $x,y,\ma,\mb$ with \emph{integer} coefficients;
\item every monomial $x^cy^d\ma^a\mb^b$ occurring in $U_n$ or $V_n$ satisfies $1\le a+b\le c+d\le n$.
\end{enumerate}
\end{theorem}

\begin{proof}
\emph{Step 1: $F$ is a double hypergeometric series.}  Since $1/\bigl((Z-1)\cdots(Z-n)\bigr)=(-1)^n/(1-Z)_n$ and $(-1)^l(2x)^la_l(\alpha)=(\tfrac12-\alpha)_l(\tfrac12+\alpha)_lx^l/l!$, the series~\eqref{eq:invfact} is
\begin{equation}\label{eq:KdF}
  F=\sum_{c,d\ge0}F_{c,d}\,x^cy^d,\qquad
  F_{c,d}=\frac{(\tfrac12-\alpha)_c(\tfrac12+\alpha)_c
        (\tfrac12-\beta)_d(\tfrac12+\beta)_d}
       {(1-Z)_{c+d}\;c!\;d!} ,
\end{equation}
an identity of formal power series in $x,y$ whose coefficients are rational functions of $Z$.

\emph{Step 2: a partial differential equation.}  From~\eqref{eq:KdF},
\[
  \frac{F_{c+1,d}}{F_{c,d}}
  =\frac{(\tfrac12-\alpha+c)(\tfrac12+\alpha+c)}{(1-Z+c+d)(c+1)}
  =\frac{(c+\tfrac12)^2-A}{(c+1+d-Z)(c+1)} .
\]
Multiplying out and comparing the coefficients of $x^{c+1}y^d$, this says exactly that
\begin{equation}\label{eq:hornpde}
  \Theta_x\bigl(\Theta_x+\Theta_y-Z\bigr)F
  =x\Bigl(\bigl(\Theta_x+\tfrac12\bigr)^2-A\Bigr)F ,
\end{equation}
and symmetrically in $y$ with $B$.

\emph{Step 3: the equation for the logarithm.}  Substitute $F=e^L$.  Then $\Theta_xF=UF$, $(\Theta_x+\Theta_y)F=TF$, $\Theta_x(TF)=(\Theta_xT+TU)F$ and $(\Theta_x+\tfrac12)^2F=(\Theta_xU+U^2+U+\tfrac14)F$.  Dividing~\eqref{eq:hornpde} by $F$ and using $A-\frac14=\ma$,
\begin{equation}\label{eq:riccati}
  ZU=\Theta_xT+UT-x\bigl(\Theta_xU+U^2+U-\ma\bigr).
\end{equation}
The coefficient of $Z^{-n}$ on the left is $U_{n+1}$.  Comparing the coefficients of $Z^0$ gives $U_1=x\ma$, and comparing those of $Z^{-n}$, $n\ge1$, gives~\eqref{eq:hornrecc}.

\emph{Step 4: integrality and support.}  Both hold for $U_1=x\ma$ and $V_1=y\mb$.  The right-hand side of~\eqref{eq:hornrecc} involves only $\Theta_x$, multiplication by $x$, sums and products; all of these preserve integer coefficients, which proves~(i) by induction.  For~(ii), call $c+d$ the degree and $a+b$ the $\mu$-degree of a monomial.  In a product both add, so the inequalities $a+b\le c+d$ and $c+d\le n$ (with $n$ additive along $U_aT_b$, $a+b=n$) are preserved; $\Theta_x$ changes neither; the multiplication by $x$ raises the degree by one, and it is applied to terms of degree at most $n$.  Hence every monomial of $U_{n+1}$ has $1\le a+b\le c+d\le n+1$.
\end{proof}

Since $T=(\Theta_x+\Theta_y)L$ and $L_n$ has no constant term, $L_n$ is recovered from $T_n$ by dividing the coefficient of each monomial $x^cy^d$ by $c+d$:
\begin{equation}\label{eq:LfromT}
  [x^cy^d]\,L_n=\frac{[x^cy^d]\,T_n}{c+d}.
\end{equation}

\begin{example}[the case $r=1$]\label{ex:r1}
Here $T_1=x\ma+y\mb$, all monomials have $c+d=1$, so $L_1=x\ma+y\mb$ and $L_1(x_0,y_0)=\frac{1+i}2\ma+\frac{1-i}2\mb$.  Its imaginary part is $\frac12(\ma-\mb)=\frac12(A-B)$, and~\eqref{eq:kappafromL} gives $\kappa_1=\frac14(A-B)$.
\end{example}

Where the denominators come from.  By Theorem~\ref{thm:hornrec} and~\eqref{eq:LfromT}, $\kappa_r$ is an integer combination of monomials $\ma^a\mb^b$ multiplied by the numbers
\begin{equation}\label{eq:sources}
  2^{-(2r-1)}\,\operatorname{Im}\frac{x_0^c\,y_0^d}{c+d},\qquad 1\le c+d\le2r-1 .
\end{equation}
Thus there are exactly three sources of denominators: the factor $2^{-(2r-1)}$; the division by $s=c+d$; and the powers of $x_0=\frac{1+i}2$, $y_0=\frac{1-i}2$.  The next lemma computes the last two exactly.

\begin{lemma}\label{lem:cost}
Let $c,d\ge0$ with $s=c+d\ge1$.  If $\operatorname{Im}(x_0^cy_0^d)\ne0$, then
\[
\nu_2\Bigl(\operatorname{Im}\frac{x_0^cy_0^d}{s}\Bigr)=-\omega(s),\qquad \omega(s)=\Bigl\lceil\frac s2\Bigr\rceil+\nu_2(s).
\]
\end{lemma}

\begin{proof}
Write $k=\min(c,d)$ and $m=|c-d|$, so that $s=2k+m$.  Since $(1+i)(1-i)=2$,
\[
  x_0^cy_0^d=2^{-s}(1+i)^c(1-i)^d=2^{-s+k}\,\eta^{m},  \qquad \eta=1+i\ \text{or}\ 1-i .
\]
Now $\eta^m=2^{m/2}e^{\pm im\pi/4}$, so $\operatorname{Im}\eta^m=\pm2^{m/2}\sin(m\pi/4)$, which is $0$ if $4\mid m$, equals $\pm2^{(m-1)/2}$ if $m$ is odd, and equals $\pm2^{m/2}$ if $m\equiv2\pmod4$.  If $m$ is odd, then $s$ is odd and
\[
  \nu_2\bigl(\operatorname{Im}x_0^cy_0^d\bigr)  =-s+k+\frac{m-1}2=-\frac{s+1}2=-\Bigl\lceil\frac s2\Bigr\rceil .
\]
If $m\equiv2\pmod 4$, then $s$ is even and $\nu_2(\operatorname{Im}x_0^cy_0^d)=-s+k+\frac m2=-\frac s2 =-\lceil s/2\rceil$.  Dividing by $s$ subtracts $\nu_2(s)$ in both cases.
\end{proof}

\begin{lemma}\label{lem:omega}
For every $r\ge1$, $\displaystyle\max_{1\le s\le 2r-1}\omega(s)\le H_r=r+\nu_2\bigl((r-1)!\bigr)$.
\end{lemma}

\begin{proof}
By Legendre's formula, $H_r=r+\bigl(r-1-s_2(r-1)\bigr)=2r-1-s_2(r-1)$.  If $s$ is odd, then $\nu_2(s)=0$ and $\omega(s)=\frac{s+1}2\le r\le H_r$.  If $s$ is even, then $s\le2r-2$ and $\nu_2(s)\le\log_2s$, so $\omega(s)\le(r-1)+\log_2(2r-2)$; with $m=r-1$ it suffices to show that
\[
  \log_2m+s_2(m)\le m .
\]
Since $s_2(m)\le\log_2m+1$, this holds as soon as $2\log_2m+1\le m$, that is for $m\ge7$.  For the remaining cases $r=1,\dots,7$ one checks directly that $H_r=1,2,4,5,8,9,11$, while $\max_{s\le2r-1}\omega(s)=1,2,4,4,7,7,8$.
\end{proof}

\medskip

With these two lemmas we can now prove Theorems~\ref{thm:odd} and~\ref{thm:coarseintro}.

\begin{proof}[Proof of Theorem~\ref{thm:coarseintro}]
By~\eqref{eq:sources}, every coefficient of $\kappa_r$ in the variables $\ma,\mb$ is an integer combination of numbers of valuation at least $-(2r-1)-\omega(s)$ with $s\le2r-1$ (Lemma~\ref{lem:cost}), hence at least $-(2r-1)-H_r=-E_r$ (Lemma~\ref{lem:omega}).
\end{proof}

\begin{proof}[Proof of Theorem~\ref{thm:odd}]
In~\eqref{eq:sources} the numbers $x_0^cy_0^d$ and $2^{-(2r-1)}$ have denominators that are powers of $2$, so odd primes can enter only through the division by $s=c+d\le2r-1$.  The odd part of any such $s$ is an odd integer not exceeding $2r-1$, hence divides $\operatorname{lcm}(1,3,\dots,2r-1)$.  This bounds the odd part of the denominator of $\kappa_r$ in the variables $(\ma,\mb)$; passing to $(A,B)$ through $\ma=A-\frac14$, $\mb=B-\frac14$ introduces only powers of $2$.
\end{proof}

\begin{remark}
Theorem~\ref{thm:coarseintro} is stated in the variables $(\ma,\mb)$, in which the proof is natural.  In these variables the bound $2^{E_r}$ need not be attained: for $2\le r\le8$ the exact exponent of $2$ in the denominator is $2r+\nu_2((r-1)!)$, which is smaller than $E_r$ by $r-1$ (see Remark~\ref{rem:besselgap}).  It is in the original variables $(A,B)$ that the exponent $E_r$ appears exactly, as in Table~\ref{tab:pattern}; this is the sharp law~\eqref{eq:sharp}, treated in Section~\ref{sec:normalform}.
\end{remark}

\section{The leading coefficient and a \texorpdfstring{$2$}{2}-adic vanishing}
\label{sec:extremal}

This section proves Theorems~\ref{thm:vanishintro} and~\ref{thm:extremalintro}.  Apart from the motivation, given first, and the application to the phase constants, given at the end, it uses nothing from Section~\ref{sec:setting}.

We begin with the motivation, which is the coefficient of the highest power $\ma^r$ in $\kappa_r$.  In the computed range ($r\le8$), $\kappa_r$ has degree exactly $r$ in $(\ma,\mb)$, and the coefficients of the pure powers $\ma^r$ and $\mb^r$ have a closed form.  Put
\[
  I_r=\int_0^1(1+u^2)^{r-1}du=\sum_{j=0}^{r-1}\binom{r-1}j\frac1{2j+1},\qquad \gamma_r=\frac{2\binom{2r}r}{(2r-1)\,16^r}.
\]

\begin{observation}\label{prop:appell}
For $r\le20$,
\begin{equation}\label{eq:toplayer}
  [\ma^r]\,\kappa_r=\gamma_rI_r,\qquad [\mb^r]\,\kappa_r=-\gamma_rI_r .
\end{equation}
\end{observation}

This was established by exact computation of $\kappa_r$ from~\eqref{eq:kappafromL}: for $r\le8$ as a polynomial in both variables, and for $r\le20$ on the line $\beta=\frac12$ (where $\mb=0$) by interpolation in $\ma$.  We expect~\eqref{eq:toplayer} to hold for every $r$, but we do not prove it.  It is the only place where the binomial sum $I_r$ enters the Jacobi problem; the task of this section is to determine $\nu_2(I_r)$.

\medskip

The first step is a closed form for the binomial sum.

\begin{lemma}\label{lem:endpoint}
For $m\ge0$ let
\[
J_m=I_{m+1}=\sum_{j=0}^{m}\binom mj\frac1{2j+1},\qquad c_m=\frac{m!}{(2m+1)!!}=\frac{2^m}{(2m+1)\binom{2m}m}.
\]
Then $(2m+1)J_m=2^m+2mJ_{m-1}$ for $m\ge1$, and
\[
  J_m=c_m\,Q_m,\qquad Q_m=\sum_{k=0}^m2^{m-k}\binom{2k}k\in\mathbb Z .
\]
\end{lemma}

\begin{proof}
Since $\frac{d}{du}\bigl[u(1+u^2)^m\bigr]=(2m+1)(1+u^2)^m-2m(1+u^2)^{m-1}$, integrating over $[0,1]$ gives $2^m=(2m+1)J_m-2mJ_{m-1}$.  The numbers $c_m$ satisfy $(2m+1)c_m=mc_{m-1}$.  Dividing the first relation by $(2m+1)c_m=mc_{m-1}$, the ratio $Q_m=J_m/c_m$ satisfies $Q_m=2Q_{m-1}+2^m/\bigl((2m+1)c_m\bigr)=2Q_{m-1}+\binom{2m}m$, with $Q_0=1$; hence $Q_m=\sum_{k\le m}2^{m-k}\binom{2k}k$.
\end{proof}

\medskip

Next we show that Theorem~\ref{thm:extremalintro} is equivalent to Theorem~\ref{thm:vanishintro}.  We need the following elementary facts about the numbers $c_k$.

\begin{lemma}\label{lem:conv}
$\nu_2(c_k)=k-s_2(k)=\nu_2(k!)$.  This is non-decreasing in $k$, because $\nu_2((k+1)!)=\nu_2(k!)+\nu_2(k+1)$, and tends to infinity.  Hence each $c_k$ belongs to $\Zt$, every tail
\[
  R_j=\sum_{k\ge j}c_k\qquad(j\ge0)
\]
converges in $\Qt$, and $\nu_2(R_j)\ge j-s_2(j)=\nu_2(j!)$.
\end{lemma}

\begin{proof}
$2k+1$ is odd and $\nu_2\binom{2k}k=s_2(k)$ by Kummer's theorem, so $\nu_2(c_k)=k-s_2(k)$, which equals $\nu_2(k!)$ by Legendre's formula.
\end{proof}

\begin{theorem}\label{thm:Freduction}
The following are equivalent:
\begin{enumerate}[\rm(a)]
\item $\nu_2(J_m)=m+\nu_2(m+1)$ for every $m\ge0$;
\item $R_0=\sum_{k\ge0}c_k=0$ in $\Qt$.
\end{enumerate}
\end{theorem}

\begin{proof}
\emph{A second expression for $J_m$.}  Since $(\tfrac12)_j/(\tfrac32)_j=1/(2j+1)$, the sum $J_m$ is the terminating hypergeometric series ${}_2F_1(-m,\tfrac12;\tfrac32;-1)$.  Pfaff's transformation ${}_2F_1(a,b;c;z)=(1-z)^{-a}{}_2F_1\bigl(a,c-b;c;\frac z{z-1}\bigr)$~\cite[Ch.~2]{AndrewsAskeyRoy} at $z=-1$ gives $J_m=2^m\,{}_2F_1(-m,1;\tfrac32;\tfrac12)$.  As $(-m)_k(1)_k2^{-k}/\bigl((\tfrac32)_kk!\bigr)=(-1)^k\binom mkc_k$, we obtain
\begin{equation}\label{eq:Jtilde}
  \widetilde J_m:=\frac{J_m}{2^m}=\sum_{k=0}^{m}(-1)^k\binom{m}{k}c_k ,
\end{equation}
and (a) is the statement $\nu_2(\widetilde J_m)=\nu_2(m+1)$ for all $m$.

\emph{Summation by parts.}  Write $n=m+1$.  Using $\binom{n-1}k=\sum_{j=0}^{k}(-1)^{k-j}\binom nj$ in~\eqref{eq:Jtilde} and exchanging the two summations (legitimate in $\Qt$ because $c_k\to0$),
\begin{equation}\label{eq:byparts}
  \widetilde J_m=\sum_{j=0}^{n}(-1)^j\binom njR_j =R_0+n\sum_{j=1}^{n}\frac{(-1)^j}{j}\binom{n-1}{j-1}R_j,
\end{equation}
where we used $\binom nj=\frac nj\binom{n-1}{j-1}$.

\emph{(b) implies (a).}  Assume $R_0=0$.  Then $R_1=-c_0=-1$ and $R_2=-c_0-c_1=-\frac43$.  In the last sum of~\eqref{eq:byparts}, the term $j=1$ equals $1$; the term $j=2$ is a multiple of $R_2/2=-\frac23\in2\Zt$; and for $j\ge3$, Lemma~\ref{lem:conv} gives $\nu_2(R_j/j)\ge\nu_2(j!)-\nu_2(j)=\nu_2((j-1)!)\ge1$.  So that sum is $1$ plus an element of $2\Zt$, a unit, and $\nu_2(\widetilde J_m)=\nu_2(n)=\nu_2(m+1)$.

\emph{(a) implies (b).}  By~\eqref{eq:Jtilde} and $c_k\to0$, the function $m\mapsto\widetilde J_m$ is given by a Mahler series, and therefore extends to a continuous function on $\Zt$~\cite[Ch.~4]{Koblitz}, whose value at $m=-1$ is $\sum_{k\ge0}(-1)^k\binom{-1}kc_k=\sum_{k\ge0}c_k=R_0$.  The integers $2^k-1$ tend to $-1$ in $\Zt$ as $k\to\infty$, and by (a) $\nu_2(\widetilde J_{2^k-1})=k\to\infty$.  By continuity, $R_0=0$.
\end{proof}

\medskip

It remains to prove the vanishing itself.  For $i\ge0$ consider the binomial moments of the sequence $(c_k)$,
\[
  M_i=\sum_{k\ge i}c_k\binom ki .
\]
By Lemma~\ref{lem:conv} these series converge in $\Qt$, with
\begin{equation}\label{eq:Mfloor}
  \nu_2(M_i)\ge i-s_2(i),
\end{equation}
and $M_0=R_0$ is the sum we want to evaluate.

\begin{theorem}\label{thm:vanish}
$M_0=\displaystyle\sum_{k\ge0}\frac{k!}{(2k+1)!!}=0$ in $\Qt$.
\end{theorem}

\begin{proof}
\emph{Step 1: a two-term recurrence for the moments.}  The relation $(2k+1)c_k=kc_{k-1}$ ($k\ge1$) gives, after shifting the index in the last sum,
\[
  \sum_kc_k(2k+1)\binom ki=\sum_kkc_{k-1}\binom ki =\sum_k(k+1)c_k\binom{k+1}i ,
\]
that is, $\sum_kc_k\bigl[(2k+1)\binom ki-(k+1)\binom{k+1}i\bigr]=0$.  The elementary identities
\[
\binom{k+1}i=\binom ki+\binom k{i-1}, \quad (k-i)\binom ki=(i+1)\binom k{i+1}, \quad (k-i+1)\binom k{i-1}=i\binom ki
\]
show that the bracket equals $(i+1)\binom k{i+1}-i\binom k{i-1}$.  Therefore
\begin{equation}\label{eq:contig}
  (i+1)M_{i+1}=iM_{i-1}\qquad(i\ge1).
\end{equation}

\emph{Step 2: the even moments are explicit multiples of $M_0$.}  Taking $i=2j-1$ in~\eqref{eq:contig} gives $2jM_{2j}=(2j-1)M_{2j-2}$, hence
\[
  M_{2j}=M_0\prod_{t=1}^j\frac{2t-1}{2t}=\frac{M_0}{4^j}\binom{2j}j,   \qquad\text{so}\qquad   \nu_2(M_{2j})=\nu_2(M_0)+s_2(j)-2j .
\]

\emph{Step 3: comparison with the a priori bound.}  By~\eqref{eq:Mfloor} with $i=2j$, $\nu_2(M_{2j})\ge 2j-s_2(2j)=2j-s_2(j)$.  Combining with Step~2,
\[
  \nu_2(M_0)\ \ge\ 4j-2s_2(j)\qquad\text{for every }j\ge1 .
\]
The right-hand side is unbounded as $j\to\infty$, so $\nu_2(M_0)=+\infty$, that is, $M_0=0$.
\end{proof}

In words: if $M_0\ne0$, the recurrence would force $\nu_2(M_{2j})$ to tend to $-\infty$, whereas the direct convergence estimate forces it to tend to $+\infty$.

\begin{proof}[Proof of Theorems~\ref{thm:vanishintro} and~\ref{thm:extremalintro}]
Theorem~\ref{thm:vanishintro} is Theorem~\ref{thm:vanish}, and Theorem~\ref{thm:extremalintro} then follows from Theorem~\ref{thm:Freduction}.
\end{proof}

\begin{remark}
The partial sums illustrate the vanishing: $\nu_2\bigl(\sum_{k<K}c_k\bigr)=9$, $20$, $48$, $99$, $200$, $401$, $802$ for $K=10$, $20$, $50$, $100$, $200$, $400$, $800$.  Since the full sum is $0$, the partial sum equals $-R_K$, and these values are just above the lower bound $K-s_2(K)$ of Lemma~\ref{lem:conv}.  For related work on $p$-adic valuations of binomial sums see~\cite{AidagulovAlekseyev,ZhangYuan}.
\end{remark}

\medskip

Finally, we apply these results to the phase constants.

\begin{corollary}\label{cor:extremal}
For every $r$ for which~\eqref{eq:toplayer} holds, in particular for $r\le20$,
\[
\nu_2\bigl([\ma^r]\kappa_r\bigr)=-2r-\nu_2\bigl((r-1)!\bigr)=-E_r+(r-1).
\]
\end{corollary}

\begin{proof}
By Theorem~\ref{thm:extremalintro} with $m=r-1$, $\nu_2(I_r)=(r-1)+\nu_2(r)$.  By Kummer's theorem $\nu_2(\gamma_r)=1+s_2(r)-4r$.  Adding, and using $\nu_2(r)=1-s_2(r)+s_2(r-1)$ and $\nu_2((r-1)!)=r-1-s_2(r-1)$, we get $-3r+1+s_2(r-1)=-2r-\nu_2((r-1)!)$.
\end{proof}

\begin{remark}[a gap of $r-1$ in the shifted variables]\label{rem:besselgap}
Let $\delta_r$ be the exponent of $2$ in the common denominator of the coefficients of $\kappa_r$ written in the variables $(\ma,\mb)$.  Theorem~\ref{thm:coarseintro} and Corollary~\ref{cor:extremal} give, whenever~\eqref{eq:toplayer} holds,
\[
  2r+\nu_2\bigl((r-1)!\bigr)\ \le\ \delta_r\ \le\ 3r-1+\nu_2\bigl((r-1)!\bigr).
\]
The two bounds differ by $r-1$.  Determining $\delta_r$ exactly is an open problem.
\end{remark}

\section{The sharp law as a property of one power series}\label{sec:normalform}

We now turn to the coefficients $\ell_r=[A^1B^0]\kappa_r$ of~\eqref{eq:elldef} and prove Theorem~\ref{thm:transferintro}.  Let us first explain where the difficulty lies.  By~\eqref{eq:sources}, $\ell_r$ is a sum of a large number of terms, indexed by the monomials $x^cy^d\ma^a\mb^b$ of $T_{2r-1}$.  The sharp law~\eqref{eq:sharp} cannot be read off from these terms: the valuation of the sum is strictly larger than the minimum of the valuations of the terms, and the same happens for every natural way we tried of grouping the terms into partial sums (for instance by the degree $c+d$).  The $2$-adic cancellation is a global property of the sum.  The construction below first performs the whole summation, in the form of a generating function, and only then extracts valuations.

\medskip

We now introduce the generating function on which the argument is based.  Let $L_n$ be as in Section~\ref{sec:setting}, and define
\begin{equation}\label{eq:gdef}
  g_n=\frac{\partial L_n}{\partial A}(x_0,y_0)\Big|_{A=B=0}\ \in\mathbb Q(i)\qquad(n\ge1).
\end{equation}
By Definition~\ref{def:kappa} and~\eqref{eq:elldef},
\begin{equation}\label{eq:ltangent}
  \ell_r=2^{-(2r-1)}\operatorname{Im}g_{2r-1}.
\end{equation}
For example $L_1=x\ma+y\mb$ gives $g_1=x_0=\frac{1+i}2$ and $\ell_1=\frac12\cdot\frac12=\frac14$.  The numbers $g_n$ grow like $(n-1)!$, so we use the exponential-type generating function (Borel transform)
\[
\Phi(t)=\sum_{n\ge1}\frac{g_n}{(n-1)!}\,t^n,\qquad
\operatorname{Im}\Phi(t)=\sum_{n\ge1}\frac{\operatorname{Im}g_n}{(n-1)!}\,t^n ,
\]
where $\operatorname{Im}\Phi$ denotes the series of the imaginary parts of the coefficients.

\begin{definitionx}\label{def:W}
\begin{equation}\label{eq:normalform}
  W(t)=\frac{\sinh2t}{t^2}\,\operatorname{Im}\Phi(t),
  \qquad
  w_m=[t^{2m}]\,W(t)\quad(m\ge0).
\end{equation}
\end{definitionx}

\noindent The first coefficients are
\[
  \sum_{m\ge0}w_mt^{2m}
  =1+\frac{t^2}{12}+\frac{t^4}{64}-\frac{143\,t^6}{80640}       +\frac{21787\,t^8}{46448640}-\cdots
\]
Only the coefficients of even powers of $t$ in $W$ are used.  (In fact $W$ is even; see Remark~\ref{rem:identityR}.)

\medskip

The proof uses two auxiliary series, whose coefficients have simple valuations.

\begin{lemma}\label{lem:weights}
Let
\[
S(t)=\frac{\sinh2t}{2t}=\sum_{k\ge0}e_kt^{2k},\qquad D(t)=\frac{2t}{\sinh2t}=\sum_{k\ge0}d_kt^{2k},
\]
so that $D(t)S(t)=1$.  Then $d_0=e_0=1$ and $\nu_2(d_k)=\nu_2(e_k)=s_2(k)$ for $k\ge1$.
\end{lemma}

\begin{proof}
We have $e_k=2^{2k}/(2k+1)!$, and Legendre's formula with $s_2(2k+1)=s_2(k)+1$ gives $\nu_2(e_k)=2k-\bigl(2k+1-s_2(k)-1\bigr)=s_2(k)$.  For $D$, the classical expansion $u/\sinh u=\sum_k(2-2^{2k})\Ber_{2k}u^{2k}/(2k)!$ (obtained from~\cite[Eq.~4.19.4]{DLMF} by $z\mapsto iu$) at $u=2t$ gives $d_k=(2-2^{2k})\Ber_{2k}2^{2k}/(2k)!$.  For $k\ge1$ we have $\nu_2(2-2^{2k})=1$, $\nu_2(\Ber_{2k})=-1$ by the von Staudt--Clausen theorem, and $\nu_2((2k)!)=2k-s_2(k)$; hence $\nu_2(d_k)=1-1+2k-(2k-s_2(k))=s_2(k)$.
\end{proof}

\medskip

We are now ready to prove Theorem~\ref{thm:transferintro}.  Put $\tau(m)=-2\nu_2\bigl((2m)!\bigr)=-4m+2s_2(m)$.

\begin{lemma}[convolutions that preserve $\tau$]\label{lem:conv-tau}
Let $(h_k)_{k\ge0}$ be rational numbers with $\nu_2(h_0)=0$ and $\nu_2(h_k)\ge s_2(k)$ for $k\ge1$; let $(a_m)_{m\ge0}$ be rational numbers and $a'_m=\sum_{k=0}^mh_ka_{m-k}$.  Fix $M\ge0$ and assume $\nu_2(a_j)\ge\tau(j)$ for all $j<M$.  Then $\nu_2(a'_j)\ge\tau(j)$ for all $j<M$, and
\[
  \nu_2(a'_M)\ge\tau(M)\iff\nu_2(a_M)\ge\tau(M),\qquad \nu_2(a'_M)=\tau(M)\iff\nu_2(a_M)=\tau(M).
\]
\end{lemma}

\begin{proof}
For $1\le k\le m$, the subadditivity $s_2(m)\le s_2(m-k)+s_2(k)$ gives
\[
  \bigl[s_2(k)+\tau(m-k)\bigr]-\tau(m)  =s_2(k)+4k+2s_2(m-k)-2s_2(m)\ \ge\ 4k-s_2(k)\ >\ 0 .
\]
Hence, if $\nu_2(a_j)\ge\tau(j)$ for $j<m$, every term $h_ka_{m-k}$ with $k\ge1$ has valuation strictly larger than $\tau(m)$, so $a'_m=h_0a_m+(\text{an element of valuation}>\tau(m))$ with $h_0$ a unit.  All assertions follow from the rule for the valuation of a sum (recalled in the Introduction).
\end{proof}

\begin{proof}[Proof of Theorem~\ref{thm:transferintro}]
By Definition~\ref{def:W}, $W(t)=S(t)\cdot2\operatorname{Im}\Phi(t)/t$.  Since $S$ is even, the coefficients of the even powers on both sides give
\[
  w_m=\sum_{k=0}^me_k\,b_{m-k},\qquad b_m=[t^{2m}]\frac{2\operatorname{Im}\Phi(t)}t=\frac{2\operatorname{Im}g_{2m+1}}{(2m)!},
\]
and conversely, since $DS=1$, $b_m=\sum_{k=0}^md_kw_{m-k}$.  By~\eqref{eq:ltangent}, $b_m=2^{2m+2}\ell_{m+1}/(2m)!$, so
\[
  \nu_2(b_m)=2m+2+\nu_2(\ell_{m+1})-\bigl(2m-s_2(m)\bigr)            =2+s_2(m)+\nu_2(\ell_{m+1}).
\]
Since $E_{m+1}=4m+2-s_2(m)$ by Legendre's formula, this shows that
\[
  \nu_2(b_m)\ge\tau(m)\iff\nu_2(\ell_{m+1})\ge-E_{m+1},
\]
and likewise with equalities.  It remains to pass between $(b_m)$ and $(w_m)$: by Lemma~\ref{lem:weights}, both $(e_k)$ and $(d_k)$ satisfy the hypotheses of Lemma~\ref{lem:conv-tau}, which, applied with $h=e$, $a=b$ and with $h=d$, $a=w$, gives the equivalence in both directions, with $M=r-1$.  Finally, $\nu_2(w_m)=\tau(m)=-2\nu_2((2m)!)$ says exactly that $((2m)!)^2w_m$ is a unit of $\Zt$.
\end{proof}

\begin{remark}[what has been gained]
The sharp law was a statement about sums over the four-index family of coefficients of $T_{2r-1}$.  After Theorem~\ref{thm:transferintro} it is a statement about the Taylor coefficients of one function of one variable, with a normalisation, $((2m)!)^2$, which does not involve the irregular quantity $E_r$.  The passage from $\Phi$ to $W$ mixes infinitely many coefficients, yet loses no $2$-adic information, because in each convolution one term strictly dominates all the others.  It is in this form that the law is proved in~\cite{Ultraspherical}.
\end{remark}

\begin{remark}[the real parts]\label{rem:identityR}
The real parts of the $g_n$ are much simpler than the imaginary parts.  In the computed range $n\le95$ one finds, exactly,
\[
  \operatorname{Re}\Phi(t)=\frac{t^2e^t}{\sinh2t},
  \qquad \operatorname{Im}g_{2k}=0 ,
\]
so that $W$ is even and $\frac{\sinh2t}{t^2}\,\Phi(t)=e^t+i\,W(t)$.  Equivalently, $\operatorname{Re}g_{2k+1}=(1-2^{2k-1})\Ber_{2k}$ and $\operatorname{Re}g_{2k}=\frac12(2k-1)\Eul_{2k-2}$, where $\Eul_{2k}$ are the Euler numbers, which are odd integers; thus $\nu_2(\operatorname{Re}g_n)=-1$ for every $n$.  These identities are proved in~\cite{Ultraspherical}; they are not used in the present paper.
\end{remark}

\section{Computations}\label{sec:evidence}

This section reports the exact computations that accompany the results above.  We begin with the verification of the sharp law.  The numbers $g_n$ were computed exactly for $n\le95$ from the recurrence of Theorem~\ref{thm:hornrec}.  Since only the derivative in $A$ at $A=B=0$ is needed, the recurrence is run with $\ma=-\frac14+\varepsilon$, $\mb=-\frac14$, discarding $\varepsilon^2$; after multiplication by $4^n$ all quantities are integers.

\begin{observation}\label{obs:range}
$\nu_2(\ell_r)=-E_r$ for $r\le48$, by direct computation of $\ell_r$; and $\nu_2(w_m)=-2\nu_2\bigl((2m)!\bigr)$ for $m\le46$, which corresponds to $r\le47$ in Theorem~\ref{thm:transferintro}.
\end{observation}

Table~\ref{tab:units} lists the first values of $u_m=((2m)!)^2w_m$ after removal of the power of $2$; property~\eqref{eq:W} asserts that this power is $2^0$, i.e.\ that $\nu_2(w_m)=-2\nu_2((2m)!)$, which is the second column.

\begin{table}[t]
\centering
\small
\begin{tabular}{@{}rrl@{}}
\toprule
$m$ & $\nu_2(w_m)$ & $u_m=\bigl((2m)!\bigr)^2w_m$\\
\midrule
$0$ & $0$   & $1$\\
$1$ & $-2$  & $1/3$\\
$2$ & $-6$  & $9$\\
$3$ & $-8$  & $-6435/7$\\
$4$ & $-14$ & $762545$\\
$5$ & $-16$ & $-18081757575/11$\\
$6$ & $-20$ & $8219311479225$\\
$7$ & $-22$ & $-82171171579421925$\\
$8$ & $-30$ & $1483464568148296358625$\\
$9$ & $-32$ & $-851695035559439777297967375/19$\\
$10$& $-36$ & $2137965664907224677330336155625$\\
\bottomrule
\end{tabular}
\caption{The numbers $u_m$ of property~\eqref{eq:W}.  Each is a quotient of two odd integers, hence a unit of $\Zt$, and $\nu_2(w_m)=-2\nu_2((2m)!)$.}
\label{tab:units}
\end{table}

\begin{remark}
An element of $\Zt$ is a unit if and only if it is congruent to $1$ modulo $2$.  With $\mathcal U(q)=\sum_{m\ge0}u_mq^m$, property~\eqref{eq:W} therefore says that $\mathcal U$ has coefficients in $\Zt$ and satisfies the single congruence $(1-q)\,\mathcal U(q)\equiv1\pmod 2$.
\end{remark}

\medskip

We finally discuss the absence of a simple equation for $W$.  A standard way to prove the integrality of the coefficients of a power series is to find a linear differential equation or recurrence that it satisfies and to study it $p$-adically, in the manner of Dwork~\cite{Dwork,Nemoto}.  Within the finite family of equations described below, we found none.

\begin{observation}\label{cert:noeq}
With $w_m$ known exactly for $m\le46$, none of the three series
\[
  \sum_mw_mq^m,\qquad   \sum_m\bigl((2m)!\bigr)^2w_mq^m,\qquad   \sum_m\frac{\operatorname{Im}g_{2m+1}}{(2m)!}q^m
\]
satisfies any of the following, with polynomial coefficients of degree at most $D$ and $(R+1)(D+1)\le35$: a linear recurrence $\sum_{i\le R}p_i(m)a_{m+i}=0$; a linear differential equation $\sum_{i\le R}p_i(q)f^{(i)}=0$; an algebraic equation $\sum_{i\le R}p_i(q)f^i=0$.  Each search was made with at least six more equations than unknowns.  The same search recovers the known equations for $\sum q^m/((2m)!)^2$, for the generating function of the Catalan numbers and for $\sum m!\,q^m$.
\end{observation}

This does not exclude that $W$ satisfies an equation of larger size.  It shows only that no relation occurs within the tested family; a relation of larger order or degree, or of a different type, is not ruled out.  The proof in~\cite{Ultraspherical} does not rely on such an equation.

\appendix

\section{An integral representation}\label{sec:mellinsplit}

This appendix is independent of the rest of the paper.  In~\cite{Ultraspherical} the numbers $g_n$ are obtained from the Gauss hypergeometric function ${}_2F_1(a,b;a+b;z)$ at the point $z=2$, which lies on its branch cut $[1,\infty)$; the value there is understood as the limit from the upper half-plane, written ${}_2F_1^{\uparrow}(a,b;a+b;2)$.  More precisely, with $\sigma=(1-Z)/2$, $a=\sigma-\delta/2$ and $b=\sigma+\delta/2$, the generating function of the $g_n$ is obtained from the second derivative of this boundary value with respect to $\delta$ at $\delta=0$.  We record an elementary formula which separates the contribution of the branch cut.  Below, $\mathrm B(p,q)=\Gamma(p)\Gamma(q)/\Gamma(p+q)$ is the Beta function.

\begin{proposition}\label{prop:mellinsplit}
Assume $0<\operatorname{Re}a<1$ and $\operatorname{Re}b>0$.  Then
\begin{align*}
 {}_2F_1^{\uparrow}(a,b;a+b;2)
 =\frac{1}{\mathrm B(a,b)}\bigg\{&
 2^{-b}\int_0^1 u^{b-1}(1-u/2)^{a-1}(1-u)^{-a}\,du\\
 &+e^{i\pi a}2^{1-a-b}\int_0^1
 u^{-a}(1-u)^{a-1}(1+u)^{b-1}\,du\bigg\},
\end{align*}
and
\begin{align*}
 \int_0^1u^{b-1}(1-u/2)^{a-1}(1-u)^{-a}\,du
 &=\mathrm B(b,1-a)\,{}_2F_1\bigl(1-a,b;\,b+1-a;\,\tfrac12\bigr),\\
 \int_0^1u^{-a}(1-u)^{a-1}(1+u)^{b-1}\,du
 &=\mathrm B(1-a,a)\,{}_2F_1\bigl(1-b,1-a;\,1;\,-1\bigr).
\end{align*}
Both sides extend meromorphically in $(a,b)$, and the identity persists.  At $\delta=0$, that is $a=b=\sigma$, the two Gauss functions are
\[
 {}_2F_1\bigl(1-\sigma,\sigma;1;\tfrac12\bigr)
 =\frac{\sqrt\pi}{\Gamma\bigl(1-\frac \sigma2\bigr)
                    \Gamma\bigl(\frac{\sigma+1}2\bigr)},
 \qquad
 {}_2F_1\bigl(1-\sigma,1-\sigma;1;-1\bigr)
 =\frac{\Gamma\bigl(\frac{3-\sigma}2\bigr)}
        {\Gamma(2-\sigma)\,\Gamma\bigl(\frac{1+\sigma}2\bigr)} .
\]
\end{proposition}

\begin{proof}
Euler's integral ${}_2F_1(a,b;a+b;z)=\mathrm B(a,b)^{-1}\int_0^1v^{b-1}(1-v)^{a-1}(1-zv)^{-a}dv$ is valid for $z$ in the upper half-plane; let $z\to2$ (dominated convergence applies under the stated assumptions) and split the integral at $v=\frac12$.  On $(0,\frac12)$ put $v=u/2$.  On $(\frac12,1)$ put $v=(1+u)/2$; there $1-zv$ tends to the negative number $-u$ from below, so $(1-zv)^{-a}\to e^{i\pi a}u^{-a}$.  This gives the first formula.  The two integrals are again Euler integrals, $\int_0^1u^{q-1}(1-u)^{c-q-1}(1-zu)^{-p}du=\mathrm B(q,c-q)\,{}_2F_1(p,q;c;z)$, with $(p,q,c,z)=(1-a,b,b+1-a,\frac12)$ and $(1-b,1-a,1,-1)$ respectively.  At $a=b=\sigma$ the first is evaluated by Gauss's second summation theorem and the second by Kummer's summation theorem~\cite[Sec.~3.1]{AndrewsAskeyRoy}.
\end{proof}

Thus the boundary value is the sum of a part without branch contribution and a part carrying the explicit phase $e^{i\pi a}$, and at $\delta=0$ both parts are quotients of Gamma functions.  The proposition also shows that replacing the upper boundary value by the lower one conjugates the phase, so the choice matters.  For background on the functions ${}_2F_1(a,b;a+b;z)$ see~\cite{AndersonEtAl,SimicVuorinen}; the reduction of the double series~\eqref{eq:KdF} to a Gauss function is of the type classified in~\cite{VidunasAppell}.

\section*{Declarations}

\noindent\textbf{Competing interests.}  The author declares no competing
interests.

\smallskip
\noindent\textbf{Data availability.}  This paper reports no experimental data.  All computations are exact, in rational or Gaussian-rational arithmetic.

\end{document}